\documentclass[11pt]{article}

\usepackage{amssymb,amsfonts,amsmath,amsthm}
\usepackage{geometry,cite}
\usepackage{color,enumerate,graphicx,hyperref}
\usepackage{stix}
\usepackage[normalem]{ulem}
\usepackage{tikz}
\usetikzlibrary{hobby}

\newcommand{\eps}{\varepsilon}
\newcommand{\R}{\mathbb R}

\newcommand{\Sphere}{\mathbb S}
\newcommand{\Z}{\mathbb Z}
\newcommand{\N}{\mathbb N}
\newcommand{\C}{C_0}

\newcommand{\E}{\mathcal{E}}

\newcommand{\m}{m}

\newcommand{\intd}{\, \mathrm{d}}

\newtheorem{theorem}{Theorem}
\newtheorem{proposition}[theorem]{Proposition}

\newtheorem{lemma}[theorem]{Lemma}
\newtheorem{corollary}[theorem]{Corollary}
\newtheorem{remark}[theorem]{Remark}

\numberwithin{equation}{section}
\numberwithin{theorem}{section}

\usepackage{todonotes}

\title{Existence of higher degree minimizers in the magnetic skyrmion problem}

\author{Cyrill B. Muratov \footnote{Dipartimento di Matematica,
    Universit\`a di Pisa, Largo Bruno Pontecorvo 5, 56127 Pisa, Italy;
    \\ Email: {\tt cyrill.muratov@unipi.it}} \and Theresa
  M. Simon\footnote{Institut f\"{u}r Analysis und Numerik,
    Universit\"{a}t M{\"u}nster, 48149 M{\"u}nster, Germany} \and
  Valeriy V. Slastikov\footnote{School of Mathematics, University of
    Bristol, Bristol BS8 1UG, United Kingdom}}

\date{\today}

\begin{document}
\maketitle

\abstract{We demonstrate existence of topologically nontrivial energy
  minimizing maps of a given positive degree from bounded domains in
  the plane to $\mathbb S^2$ in a variational model describing
  magnetizations in ultrathin ferromagnetic films with
  Dzyaloshinskii-Moriya interaction.  Our strategy is to insert tiny
  truncated Belavin-Polyakov profiles in carefully chosen
  locations of lower degree objects such that the total energy
  increase lies strictly below the expected Dirichlet energy
  contribution, ruling out loss of degree in the limits of minimizing
  sequences.  The argument requires that the domain be either
  sufficiently large or sufficiently slender to accommodate a
  prescribed degree. We also show that these higher degree minimizers
  concentrate on point-like skyrmionic configurations in a suitable
  parameter regime. \\

  
  }

\section{Introduction}

Topological solitons are a central notion for a number of nonlinear
field theories arising as mathematical models of systems of very
different physical nature \cite{manton}. Broadly speaking, they are
certain special solutions of nonlinear partial differential equations
of a field theory in the whole space that, on one hand, are in a
certain sense localized and, on the other hand, exhibit a certain
degree of persistence among more general classes of
solutions. Topological solitons constitute the backbone of {\em
  topological defects}, which, in turn, are stable localized nonlinear
excitations of the topologically trivial background state in a
nonlinear system. The stability of these defects is closely related to
the topological character of topological solitons through the fact
that they cannot be smoothly deformed to the background state due to
topological obstruction.

A prime example of topological defects are Abrikosov vortices in
type-II superconductors, which may be described by the Ginzburg-Landau
theory \cite{tinkham}. Mathematically the vortex solution can already
be captured by the single Ginzburg-Landau equation for a
complex-valued field in the plane, with the individual vortex solution
in $\mathbb R^2$ providing an example of a topological
soliton. Starting with the studies in the applied mathematics
literature \cite{greenberg80,hagan82,neu90} (this and the subsequent
lists of references are not intended to be exhaustive), existence and
uniqueness of equivariant solutions (``radial'' solutions with
prescribed degree $d \in \mathbb Z$) was established in
\cite{chen94,herve94}. Furthermore, these solutions with $d = \pm 1$
were shown to be the only non-trivial locally minimizing solutions in
the sense of De Giorgi for the associated energy
\cite{mironescu96,sandier98}.

As for the topological defects within the Ginzburg-Landau theory, the
Dirichlet boundary data of non-trivial topological degree force the
existence of minimizers (of the same degree) exhibiting point-like
vortices in the domain interior, where the energy concentrates as a
coherence length parameter goes to zero \cite{bethuel}. One can obtain
a lot of information about minimizing solutions of the Ginzburg-Landau
energy, including locations and degrees of vortices, expansion of the
energy with respect to the coherence length parameter and fine
properties of minimizers
\cite{bethuel,mironescu95,jerrard02,pacard}. The connection between
the vortex solutions with degree $d = \pm 1$ in the whole plane and
the blowup limits of topological defects was established in
\cite{shafrir95}. There is a vast literature on the subject and many
questions are still unresolved, for further references and some open
problems see \cite{pacard,sandier,brezis23}.

We note that the whole space Ginzburg-Landau vortex solutions do not
actually represent global energy minimizers with prescribed
topological degree, as the Dirichlet energy contribution of these
solutions diverges logarithmically at infinity (however, compare with
\cite{almog09}). In search of the genuine global energy minimizing
topological solitons for field theories, various stabilization
mechanisms have been considered, starting with the model proposed by
Skyrme \cite{skyrme62}. For that model, existence of $\mathbb S^3$
valued topologically nontrivial energy minimizers in $\mathbb R^3$,
termed {\em skyrmions}, was established for degrees $d = \pm 1$
\cite{esteban86,esteban90,esteban92,esteban04,lin04a}. Another variant
of the Skyrme model in $\mathbb R^3$ was investigated in
\cite{faddeev97,lin04a,battye99}, where minimizers were found to exist
for an infinite subset of degrees $d \in \mathbb Z$. In $\mathbb R^2$,
Skyrme model with additional energy terms yields the so-called {\em
  baby skyrmions} as maps from the plane to $\mathbb S^2$ with degree
$d = \pm 1$ \cite{lin04,li11}. More recently, strong numerical
evidence was provided for the existence of {\em hopfions} as locally
energy minimizing maps from $\mathbb R^3$ to $\mathbb S^2$
\cite{rybakov22}, following the early work in \cite{bogolubsky88}, for
the energy containing higher derivative penalty terms in addition to
the Dirichlet energy.

Recently, a growing body of work has emerged with the studies of
chiral magnetic skyrmions, or simply {\em magnetic skyrmions} for
shorthand, motivated by the experimental discovery of these
configurations in chiral magnets and ultrathin ferromagnetic
heterostructures \cite{muhlbauer09,yu10,heinze11,romming13}. In
these systems, the skyrmion solutions are typically stabilized by a
chiral energy term called the Dzyaloshinskii-Moriya interaction (DMI),
which promotes rotations of the magnetization vector with values in
$\mathbb S^2$. Studies in the physics literature identified skyrmion
configurations as locally minimizing solutions of the micromagnetic
energy
\cite{bogdanov89,bogdanov89a,bogdanov94a,bogdanov99,rohart13,leonov16},
which makes them attractive candidates for information technology
applications \cite{kiselev11,nagaosa13,fert17,zhang20}.

Mathematical studies of chiral magnetic skyrmions go back to
\cite{melcher14}, which treated a model similar to the one in
\cite{lin04} and in which the original Skyrme term is replaced by a
DMI term appropriate for non-centrosymmetric cubic
materials.\footnote{In bulk chiral materials, the simplest form of the
  DMI energy density is given by a term proportional to
  $\m \cdot (\nabla \times \m)$, where $\m = (\m_1, \m_2, \m_3)$ and
  $\nabla = (\partial_1, \partial_2, \partial_3)$
  \cite{nagaosa13}. Although this is different from the form
  appropriate for ultrathin ferromagnetic heterostructures in which
  the DMI is of interfacial origin and its energy density is
  proportional to
  \mbox{$\m_3 \nabla' \cdot \m' - \m' \cdot \nabla' \m_3$}, where
  $\m = (\m', \m_3)$, $\m' = (\m_1, \m_2)$ and
  $\nabla' = (\partial_1, \partial_2)$ \cite{rohart13}, when
  $\m = \m(x_1, x_2)$ these two terms are equivalent up to a
  90$^\circ$ rotation around the $x_3$-axis, since
  $m \cdot (\nabla \times \m) = \tilde\m_3 \nabla' \cdot \tilde \m' -
  \tilde \m' \cdot \nabla' \tilde \m_3$ for $\tilde \m_1 = m_2$,
  $\tilde \m_2 = -\m_1$ and $\tilde \m_3 = \m_3$. For this reason, we
  can interpret the results of the studies of bulk chiral materials in
  two dimensions in terms of models of ultrathin ferromagnetic
  heterostructures.}  This paper established the existence of degree
$d = 1$ (in our convention) global energy minimizers in the whole
plane. Further studies of these and related minimizers can be found in
\cite{doring17,li18,komineas20,komineas21,gustafson21,bms:pnas22}. In
a model that is appropriate for ultrathin multilayer materials with
interfacial DMI existence results were obtained in
\cite{bms:arma21,bms:prb20,bfbsm:prb23}, which also incorporated the
non-local stabilizing effects of the stray field (see also the
ansatz-based and numerical studies in \cite{buttner18}). Degree
$d = 1$ minimizers were constructed in bounded domains in the plane
under confinement in \cite{mmss:cmp23}. Also, precise asymptotic
characterizations of the degree $d = 1$ energy minimizing solutions in
the conformal limit, in which the energy is asymptotically dominated
by the Dirichlet energy, have been obtained, showing that the energy
minimizers approach some particular shrinking Belavin-Polyakov (BP)
profiles \cite{doring17,bms:arma21, gustafson21, mmss:cmp23}.

\begin{figure}
  \centering
  \includegraphics[width=12cm]{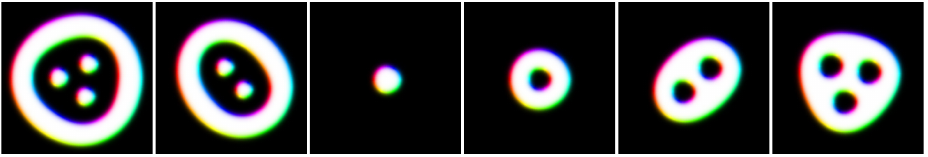}
  \caption{A series of numerical solutions of \eqref{eq:EL} with
    topological degree $d = 3, 2, 1, 0, -1, -2$, from left to right,
    obtained in \cite{rybakov19} (reproduced with permission). Black
    and white regions show the domains where $\m$ is predominantly
    down or up, respectively; the color indicates the direction of the
    in-plane component $\m'$ for intermediate values of $\m_3$. For
    $d = 1$, the direction of $\m'$ is parallel to that of the
    gradient of $\m_3$, which is the characteristic of the radial
    skyrmion solution.}
  \label{f:rybsols}
\end{figure}

In its simplest form, the model describing the magnetization
configurations $\m : \mathbb R^2 \to \mathbb S^2$ in ultrathin
ferromagnetic layers with perpendicular magnetic anisotropy and the
interfacial DMI starts with the energy functional in the form of
  the sum of the exchange (Dirichlet), DMI and the anisotropy energy
  terms:
\begin{align}
  \label{eq:Eintro}
  E(\m) = \int_{\mathbb R^2} \left( |\nabla \m|^2 +
  \kappa (\m_3 \nabla \cdot \m - \m' \cdot \nabla \m_3 ) + (Q - 1)
  |\m'|^2 \right) \intd x,
\end{align}
where for $m = (m_1, m_2, m_3)$ we use the convention $m = (m',m_3)$,
in which $m' = (m_1, m_2)$ is the in-plane component of the
magnetization. Here $\kappa \in \mathbb R$ and $Q \geq 1$ are
dimensionless material parameters (the DMI constant and the material's
quality factor respectively, see \cite{rohart13,bms:prb20} for the
explanation). The associated Euler-Lagrange equation can be easily
shown to be (see Proposition \ref{p:reg})
\begin{align}
  \label{eq:EL}
  \Delta \m
  & + \m  |\nabla \m|^2 + (Q - 1) (\m_3 e_3 - \m_3^2 \m) \notag
  \\
  & - 
    \kappa [(e_3 - \m_3 m) \nabla \cdot \m' - \nabla \m_3 + (\m' \cdot
    \nabla \m_3) \m ] = 0, 
\end{align}
distributionally, where the pure gradient term is understood as
$\nabla m_3 = (\partial_1 m_3, \partial_2 m_3, 0)$.  As this equation
is an $L^2_\mathrm{loc}(\mathbb R^2)$ perturbation of the harmonic map
equation from $\R^2$ to $\mathbb S^2$, its energy minimizing solutions
are known to be smooth \cite[Chapter 4]{moser}. We remark that in the
simplest case of the parameters $Q = 1$ and $\kappa = 0$ equation
\eqref{eq:EL} is just the harmonic map equation, whose solutions with
bounded energy on the whole of $\R^2$ had been completely
characterized \cite{lemaire78,wood74,eells78}. They are minimizers of
the energy in their respective homotopy classes determined by the
topological degree \cite{brezis83a,brezis95}
\begin{align}
  \label{eq:dintro}
  d = {1 \over 4 \pi} \int_{\R^2} \m \cdot (\partial_1 \m \times
  \partial_2 \m) \intd x \in \mathbb Z, 
\end{align}
which were first constructed in \cite{belavin75}. However, these
solutions do not qualify as topological solitons due to the conformal
invariance of the Dirichlet energy in $\R^2$ and hence the absence of
a common characteristic length scale. We also note that for
$\kappa = 0$ and $Q > 1$ equation \eqref{eq:EL} has no non-trivial
solutions by the Derrick-Pohozaev argument \cite{derrick64}, while for
$Q \geq 1$ and $\kappa \not= 0$ sufficiently large \eqref{eq:EL}
exhibits solutions in the form of spin spirals whose energy diverges
to $-\infty$ \cite{rohart13,ms:prsa17}. In contrast, for $Q > 1$ and
$\kappa \not= 0$ sufficiently small there is always a solution with
degree $d = 1$ that converges to $\m = -e_3$ at infinity and for whose
existence the DMI term is indispensable
\cite{bms:prb20,bms:arma21,gustafson21}. We remark, however, that at
the same time the interplay between the DMI energy and the topological
degree $d \not=1$ appears to be far from straightforward for this type
of profiles.

Numerical studies of \eqref{eq:EL} reveal a wealth of locally energy
minimizing solutions in $\mathbb R^2$ for various values of the
topological degree
\cite{bogdanov99,rybakov19,foster19,kuchkin20,kuchkin23}. For a sample
of the observed numerical solutions, see Fig. \ref{f:rybsols}. In
particular, the problem turns out to be considerably richer than its
Ginzburg-Landau counterpart, exhibiting a plethora of solutions beyond
a simple equivariant ``radial'' form first studied in
\cite{bogdanov89a,bogdanov94a}. Furthermore, much less is known about
the topologically nontrivial globally minimizing configurations. For
example, uniqueness and radial symmetry of the $d = 1$ minimizers for
$Q > 1$ are not known and are only asymptotically obtained in the
conformal limit $\kappa \to 0$ \cite{bms:arma21}. Furthermore, with
the exception of some very special choices of models yielding explicit
solutions via the Bogomolnyi trick \cite{barton-singer20} (see also
\cite{ibrahim23,hill21}), no existence for any other degree
$d \not= 1$ has been known up to now, not even under confinement. This
may be contrasted, for example, with the available results in
\cite{lin04a} in which an infinite subset of degrees (possibly all of
$\mathbb Z$) yields existence, and all degrees $d \in \mathbb Z$ yield
minimizers under confinement. This is because in the Skyrme mechanisms
the higher-order term in the energy prevents concentration and
collapse of the minimizing sequences and a subsequent loss of the
degree in the limit. In contrast, even under confinement the energy in
\eqref{eq:Eintro} generally allows for concentration \cite{lin99}, and
the question of existence of minimizers with prescribed degree is
genuinely non-trivial.

In this paper, we investigate existence of energy minimizing solutions
of \eqref{eq:EL} with prescribed degree $d \geq 1$ under confinement
in a bounded domain $\Omega \subset \R^2$ subject to the Dirichlet
boundary condition $\m = -e_3$ on $\partial \Omega$. This formulation
was used in our earlier work \cite{mmss:cmp23} to study degree $d = 1$
single skyrmion solutions and is relevant to ultrathin film
ferromagnetic materials in suitable parameter regimes
\cite{dms:mmmas24}. It is well suited for the study of multiple
skyrmions, similarly to the Dirichlet problem for Ginzburg-Landau
vortices \cite{bethuel}.  In terms of the energy minimizing
configurations for \eqref{eq:Eintro}, our paper is the first to
establish existence of higher degree magnetizations in the context of
multiple magnetic skyrmions on large, bounded domains for $Q > 1$. We
also establish existence of minimizers with higher degrees for $Q = 1$
for sufficiently slender domains characterized in terms of the
domain's optimal Poincar\'e constant.

Our existence results open up the question of how skyrmions
interact. At this point, it is not even clear whether multiple
degree-one skyrmions or single high-degree skyrmions could develop.
We do not address this issue here, which would require an analysis of
the splitting alternative in the concentration compactness on the
whole of $\R^2$. Our proof focuses instead on ruling out the vanishing
alternative in the concentration compactness on bounded domains. It
relies on a careful construction, inductively inserting a tiny,
truncated BP profile in a location where the degree $d-1$ minimizer is
almost constant and making sure that the energy increases by strictly
less than the additional contribution in the exchange energy.  This
estimate then allows us to rule out loss of degree in weak limits of
minimizing sequences.

A surprising amount of care needs to be taken in choosing the location
for insertion as a result of the rigidity of the harmonic map problem:
The error terms in the exchange energy when pasting together the lower
degree minimizer and the BP profile need to be dominated by the
gain in the DMI energy, which is of lower order.  Ideally, one would
thus choose a location where the local exchange contribution is small
at lower order.  We show this to be possible when our domain $\Omega$
is sufficiently large or sufficiently slender, using an appropriate
covering argument to handle the problem. Our existence result is
presented in Theorem \ref{thm:existence}. We note that numerical
evidence suggests that a given domain can only support minimizers with
the degree bounded above depending on the domain geometry.

As it stands, our existence result so far yields very little
information on the structure of the obtained solutions. As was already
noted, it would be natural to ask whether these minimizers may indeed
be interpreted as topological defects consisting of multiple
well-separated skyrmions. In analogy with the Ginzburg-Landau problem,
we therefore consider the limit behavior of the minimizers in which
the anisotropy term in the energy forces the magnetization to converge
to $\m = -e_3$ almost everywhere, which is achieved by sending the
parameter $Q$ to infinity. In this limit, we can prove that the
Dirichlet energy density of the minimizers concentrates on a quantized
atomic defect measure, see Theorem \ref{thm:asymptotics}. However, our
convergence result gives no further information on either the location
of the limit measure support or its quantized amplitudes, which
correspond to multiples of the degrees of the associated shrinking
bubbles.  At the heart of this issue lies the question of
\emph{interaction} between multiple skyrmions: Do they repel and
settle in distinct locations, or do they coalesce into a genuinely
higher degree object?  Mathematically, this requires a better
understanding of the rigidity properties of higher degree harmonic
maps expected to arise as the blowup limits above and their interplay
with the lower order terms in our energy, as the higher degree
situation \cite{rupflin23} is significantly more complex than in the
degree one case \cite{bms:arma21,hirsch21,topping21}.

The remainder of our paper is organized as follows. In Section
\ref{s:state}, we give the precise mathematical formulation of the
problem, state our main theorems and discuss how their conclusions
depend on various ingredients of the problem. In this section we also
give an outline of the arguments used in the proofs. In Section
\ref{sec:auxiliary-results}, we establish several technical results
that provide key ingredients in the proofs of the main
theorems. Finally, in Section \ref{s:proofs} we conclude the proofs.

\paragraph{Acknowledgements.}

C. B. Muratov was supported by MUR via PRIN 2022 PNRR project
P2022WJW9H and acknowledges the MUR Excellence Department Project
awarded to the Department of Mathematics, University of Pisa, CUP
I57G22000700001. The work of TMS was funded by the Deutsche
Forschungsgemeinschaft (DFG, German Research Foundation) under
Germany's Excellence Strategy EXC 2044 – 390685587, Mathematics
Münster: Dynamics--Geometry--Structure. C. B. Muratov is a member of
INdAM-GNAMPA.

\section{Statement of results}
\label{s:state}

We now give the precise mathematical statements of our results and
  outline our strategy of their proof.

\subsection{Mathematical setup}

Following the setup of our paper \cite{mmss:cmp23} on single skyrmions
on a bounded domain $\Omega \subset \R^2$ with Lipschitz boundary,
we wish to minimize the energy in \eqref{eq:Eintro} restricted to
  $\Omega$ under Dirichlet boundary condition $m = -e_3$ on
  $\partial \Omega$ \cite{dms:mmmas24}, which after an integration by parts
  can be equivalently defined as \cite{bms:arma21}
\begin{align}\label{def:energy}
  \E (\m) := \int_{\Omega} \left( |\nabla \m|^2 -2 \kappa 
  \m'\cdot \nabla m_3   + (Q-1) |\m'|^2 \right) \intd x. 
\end{align}
Passing from $\m = (m', m_3)$ to $\tilde \m := (-m', m_3)$ and noting
that the energy remains unchanged when replacing $m$ by $\tilde m$ and
changing the sign of $\kappa$, throughout the rest of the paper we may
assume without loss of generality that $\kappa >0$. 

For the energy in \eqref{def:energy} and a given $d \in \mathbb Z$ we
consider the set of admissible functions
\begin{align}
  \mathcal A_d := \left\{ \m \in H^1(\Omega; \Sphere^2) \ : \ \m= -
  e_3 \hbox{ on } \partial \Omega , \ \mathcal{N}(\m) =d \right\},
\end{align}
which satisfy a specific Dirichlet boundary condition and whose
topological degree $\mathcal N(m)$ is equal to $d$.  The degree of a
function $\m \in \mathring H^1(\R^2;\Sphere^2)$, where, as usual,
\begin{align}
    \mathring H^1(\R^2, \Sphere^2) := \left\{m \in H^1_{\mathrm{loc}}
  (\R^2; \R^3) :
  \int_{\R^2} |\nabla m|^2 \intd x < \infty, \ | m | = 1 \
  \text{a.e. in } \R^2 \right\},
\end{align} 
can be defined as \cite{brezis83a,brezis95}
\begin{align}
  \label{eq:N}
  \mathcal N (\m) =\frac{1}{4 \pi} \int_\Omega \m \cdot (\partial_1 \m
  \times \partial_2 \m) \intd x.
\end{align}
It is well known that $\mathcal N(\m) \in \Z$ for any
$\m \in \mathring H^1(\R^2;\Sphere^2)$, see \cite{brezis83a}.  To
apply this definition to our case of the bounded domain $\Omega$, we
extend $m\in \mathcal{A}_d$ to the whole of $\R^2$ by setting
$m = -e_3$ outside $\Omega$. Indeed, in the rest of this paper we will
not distinguish between $\m \in \mathcal A_d$ and its extension to the
whole of $\R^2$.

We furthermore denote the smallest Dirichlet eigenvalue of the domain
$\Omega$ associated with the optimal Poincar{\'e} constant of $\Omega$
by $\lambda_0 >0$. Hence for all $f\in W^{1,2}_0(\Omega)$, we have
\begin{align}
  \int_\Omega |f|^2 \intd x \leq \lambda_0^{-1} \int_\Omega |\nabla
  f|^2 \intd x. 
\end{align}
Throughout the rest of the paper, $C>0$ denotes a generic, universal
constant which may change from line to line, unless specified
otherwise. For simplicity, we also use the notation $B_r$ to
  denote the open ball of radius $r$ centered at the origin. 

\subsection{Main results}

Our main result is formulated in the following theorem, giving a
condition for existence of higher degree minimizers in terms of the
quantity
\begin{align}\label{def:alpha}
  \alpha(Q,\kappa) & :=  \frac{2 \kappa^2}{\sqrt{(Q-1)^2 +
                     4\lambda_0\kappa^2}+(Q-1)} 
\end{align}
being small enough and the area of $\Omega$ being large enough
compared to $\frac{d}{\kappa^2}$, see
Theorem~\ref{thm:existence}. Note that in both cases $Q=1$ and $Q>1$
we need $\kappa$ to be small compared, respectively, to
$\sqrt{\lambda_0}$ and $\sqrt{Q-1}$. This can be seen from the lower
bounds on $\alpha(Q,\kappa)$ in Lemma~\ref{lemma:alpha}. Notice also
that the quantity $\alpha(Q, \kappa)$ implicitly depends on $\Omega$
via the value of $\lambda_0$.

\begin{theorem}\label{thm:existence}
  Let $\Omega \subset \R^2$ be a bounded domain with Lipschitz
  boundary. There exists a universal constant $\overline{C}>0$ with
  the following property:
Let $\kappa >0$, $Q\geq 1$, and $d \in \N$ with
\begin{align}\label{eq:smallness_kappa_d}
  	\alpha(Q,\kappa) &
	 \leq \min\left\{ \frac{2}{d+1}, \frac{1}{2} \right\}.
\end{align}
If
\begin{align}\label{eq:cond_omega_big}
  |\Omega| \geq \frac{\overline{C} d}{\kappa^2},
\end{align}
then there exists a minimizer of
  $\E$ over $\mathcal{A}_d$.
\end{theorem}

A standard consequence of the minimality of $\E$ is that the minimizer
solves the Euler-Lagrange equation in \eqref{eq:EL}. More precisely,
we have the following result.

\begin{proposition}
  \label{p:reg}
  Under the assumptions of Theorem \ref{thm:existence}, let $m$ be a
  minimizer of $\E$ over $\mathcal A_d$. Then
  $m \in C^\infty(\Omega; \R^3)$, $|m| = 1$ in $\Omega$, and $m$
  satisfies \eqref{eq:EL} classically in $\Omega$. If, furthermore,
  $\Omega$ is a simply connected, bounded, open set with boundary of
  class $C^{1,\alpha}$ for some $\alpha > 0$, then
  $m \in C^\infty(\Omega; \R^3) \cap C(\overline \Omega; \R^3)$.
\end{proposition}
\noindent We also note that the above proposition applies to every
critical point of $\E$ in $\mathcal A_d$ regardless of the choice of
parameters.  In fact, as can be seen from the proof of Proposition
\ref{p:reg}, just stationarity with respect to smooth outer variations
is sufficient, as in two-dimensional domains equation \eqref{eq:EL}
has a good regularity theory.
 
The statement of Theorem \ref{thm:existence} may be simplified in two
extreme cases to yield a more explicit dependence on the domain
$\Omega$. The first case corresponds to the value of $Q >1$ fixed and
the domain being sufficiently large, so that the value of $\lambda_0$
is negligible in the definition of $\alpha(Q, \kappa)$. Since
\begin{align} \label{est:al}
    \begin{split}
      \alpha(Q,\kappa) \leq \frac{\kappa^2}{Q-1}
    \end{split}
\end{align}
for all $Q > 1$ and $\kappa > 0$, Theorem \ref{thm:existence}
immediately yields the following result.
 
\begin{corollary} \label{c:Q} Let $d \in \mathbb N$, $Q > 1$ and
  $0 < \kappa < \sqrt{(Q - 1) \min \left\{ {2 \over d + 1}, \frac12
    \right\} }$. Then there exists a minimizer of $\E$ over
  $\mathcal A_d$ for all
  $|\Omega| \geq \frac{\overline{C} d}{ \kappa^2}$, for some
  $\overline{C} > 0$ universal.
\end{corollary}
This corollary implies, in particular, that given $d \in \mathbb N$
and $Q > 1$ fixed, given $\kappa > 0$ sufficiently small depending on
$d$ and $Q$, and given a bounded domain $\Omega_0 \subset \R^2$, there
exists a scale factor $s_0 > 0$ sufficiently large depending on $d$,
$\kappa$ and $\Omega_0$ such that for all $s > s_0$ and all
$\Omega = s \Omega_0$ the energy $\E$ admits a minimizer in the
admissible class $\mathcal A_d$. In particular, this is true for any
$d \in \mathbb N$ in the simplest case of $\Omega_0$ being a
disk.

On the other hand, for $d \gg 1$ the corollary gives existence only
for $\kappa < C \sqrt{(Q - 1) / d}$, implying that we must have
$|\Omega| \geq C' d^2 / (Q - 1)$, for some $C, C' > 0$
universal. Notice that this scaling of $|\Omega|$ with $d$ is
consistent with configurations in which the skyrmions line the domain
boundary, one skyrmion per natural length scale $l = 1/\sqrt{Q - 1}$
of a single skyrmion. Nevertheless, one would rather expect that the
skyrmions fill the interior of the domain $\Omega$ uniformly, which
would yield a different scaling $|\Omega| \geq C d / (Q - 1)$, for
some $C > 0$ universal. Our methods are too coarse to discriminate
between these two possibilities, which would require to rule out
preferential placement of skyrmions close to the domain boundary.

In the opposite extreme of $Q = 1$ we have instead the following
corollary of Theorem \ref{thm:existence}.

\begin{corollary} \label{c:1} Let $d \in \mathbb N$, $Q = 1$ and
  $0 < \kappa < \sqrt{\lambda_0} \min \left\{ {2 \over d + 1}, \frac12
  \right\}$. Then there exists a minimizer of $\E$ over $\mathcal A_d$
  for all $|\Omega| \geq \frac{\overline{C} d }{ \kappa^2}$, for some
  $\overline{C} > 0$ universal.
\end{corollary}

We note that in the case $Q = 1$ we cannot ensure existence of
minimizers by rescaling $s\Omega$ with some large scale factor $s > 0$
since doing so decreases $\lambda_0$, which feeds back into the
smallness condition for $\kappa$.  In particular, existence of
minimizers depends on the \emph{shape} of $\Omega$, not only on its
area as in Corollary \ref{c:Q}. This point may be illustrated by
considering $\Omega$ to be a strip of width 1 and varying length
$L \geq 1$, i.e.  $\Omega = (0,L) \times (0, 1)$. Clearly in this case
the value of $\lambda_0$ may be bounded below by a universal
  constant uniformly in $L$. Thus, we have
existence for all $0 < \kappa < C / d$ for some $C > 0$ universal and
$L\geq L_0$ with $L_0 =\overline{C} d/\kappa^2$. However, our
methods do not allow to conclude whether there is a minimizer of $\E$
over $\mathcal A_d$ for any $d > 1$ in the case $Q = 1$ and
$L = 1$, no matter what the value of $\kappa$ is, as there is
currently no reasonable quantitative information on the value of the
universal constant $\overline{C}$ in the statement of Theorem
\ref{thm:existence}. Similarly, contrary to the case $Q > 1$ we do
not know if Corollary \ref{c:1} ever applies with $d > 1$ and $Q = 1$
when $\Omega$ is a disk of any given radius. This should be
contrasted with the result in \cite{mmss:cmp23} for $d = 1$, which
yields existence of minimizers in this case for an arbitrary domain
$\Omega$, provided $\kappa$ is sufficiently small. Also notice that
for $d \gg 1$ and $\Omega$ in the form of the strip as above,
Corollary \ref{c:1} only yields existence for $L \geq C d^3$ with some
$C > 0$ universal, while we expect that existence may hold as soon as
$L \geq C d$.

We now let $d \in \mathbb N$, $\kappa >0$ and the domain $\Omega$
satisfying condition \eqref{eq:cond_omega_big} be fixed and consider
the limit $Q\to \infty$, in which minimizers of $\E$ over
$\mathcal A_d$ are expected to concentrate. In view of \eqref{est:al},
we have existence for any $Q > 1$ sufficiently big.  We may thus
analyze the asymptotic behaviour of the minimizers in this regime.
For technical reasons, we additionally assume $\Omega$ to be simply
connected and have a $C^{1,\alpha}$ boundary.
\begin{theorem}\label{thm:asymptotics}
  Let $d\in \N$ and $\kappa > 0$, and let $\Omega \subset \R^2$ be a
  bounded, open, simply connected domain with a $C^{1,\alpha}$
  boundary for some $\alpha > 0$, satisfying
  \begin{align}\label{con:om}
    |\Omega| \geq {\frac{\overline{C}  d}{\kappa^2}},
  \end{align}
  where $\overline{C}$ is as in Theorem \ref{thm:existence}.  Then
  for each $Q > 1$ large enough there exists a minimizer of $\E$ over
  $\mathcal A_d$.  Furthermore, if $(Q_n)$ is a sequence such that
  $Q_n \to \infty$ as $n \to \infty$, and $m_n$ is a minimizer of $\E$
  with $Q = Q_n$ over $\mathcal A_d$, there exist $k\in \N$ with
  $k\leq d$, $x_1,\ldots, x_k \in \overline{\Omega}$ distinct and
  $d_1,\ldots, d_k \in \N$ with $\sum_{j=1}^k d_j = d$ such that as
  $n \to \infty$ we have
  \begin{align}
    m_n & \rightharpoonup -e_3 \qquad \text{in} \
          W^{1,2}(\Omega;\R^3), 
  \end{align}
  and for
  $d \mu_n := \left( |\nabla m_n|^2 - 2 \kappa m' \cdot \nabla m_3 + (Q
    - 1) |m'|^2 \right) \intd x$ there holds
  \begin{align}
    \label{eq:8piddxj}
    \mu_n \stackrel{*}{\rightharpoonup}
    \sum_{j=1}^k 8\pi {d}_j \delta_{x_j} 
  \end{align}
  in the sense of measures, possibly up to extraction of a
  subsequence.
\end{theorem}

In other words, this theorem says that for $d \in \mathbb N$,
$\kappa > 0$ and the domain $\Omega$ all fixed, as $Q \to \infty$ the
energy density of minimizers over $\mathcal A_d$ concentrates on a sum
of $k \in \mathbb N$ quantized delta-measures supported at points
$x_j$ in the closure of $\Omega$. Their amplitudes $8 \pi d_j$
correspond to the energies of the harmonic maps with degrees
$1 \leq d_j \leq d$ in the whole of $\R^2$, indicating a bubbling
phenomenon. In that sense one can interpret the minimizers for
$Q \gg 1$ as collections of $k$ well-separated skyrmionic
configurations. However, unless $d = 1$ we cannot conclude that
$k = d$ or, equivalently, that all $d_j = 1$ for all
$1 \leq j \leq k$, as is suggested by the results of numerical
simulations in this regime. In other words, based on the above result
we cannot conclude that minimizers with degree $d > 1$ and $Q \gg 1$
resemble a collection of $d$ well-separated skyrmions (i.e.,
minimizers with degree $d = 1$ in $\R^2$). The latter would require a
much finer analysis of the interaction of skyrmions that goes well
beyond the scope of the present paper.

\subsection{Strategy of the proof}

We apply the direct method of calculus of variations to establish
existence of minimizers of $\E$ over $\mathcal A_d$ for a given value
of $d \in \mathbb N$ in Theorem \ref{thm:existence}. As we work on a
bounded domain with Dirichlet boundary conditions,
lower-semicontinuity and coercivity of the energy functional easily
follow by the Sobolev embedding and the Cauchy-Schwarz inequality,
respectively. Therefore, the only issue in the proof of existence of
minimizers is to ensure that the degree $d$ is preserved when passing
to the weak limit of minimizing sequences. This can be achieved by
establishing a sort of strict subadditivity of the energy with respect
to the degree, which is the subject of the lemmas in Section
\ref{sec:auxiliary-results}.

More precisely, the main point in our existence argument is proving
that for each incremental increase in degree the infimal energy is
increased by strictly less than $8\pi$, which is the energy
contribution of a shrinking BP profile (a harmonic bubble of
  degree 1). This is the content of Lemma \ref{lemma:construction}.
In analogy to the existence of a single skyrmion \cite{melcher14}, we
can then rule out loss of degree in weak limits of minimizing
sequences. For $\kappa$ sufficiently small depending on $d$, an
increase in degree is prevented by the coercivity properties of the
energy presented in Lemma \ref{lemma:a_priori}. We note that in the
considered regime the energy is qualitatively dominated by the
exchange energy term.

We prove the required upper bound on the infimal energies by
constructing a competitor, in which we insert a tiny, truncated BP
profile at a point around which the lower degree minimizer is very
close to $-e_3$ in the $H^1$-topology.  As we already mentioned, this
is a surprisingly delicate issue as the exchange energy incurred by
pasting in the BP profile needs to be dominated by the gain in the DMI
energy, which is of order $\kappa^2$ (at least for $Q=1$, this is
sharp).  Therefore, we need to insert the skyrmion in a location where
the exchange energy is small at order $\kappa^2$. Surprisingly, this
only seems to be possible with further assumptions on $\Omega$ and the
parameters appearing in the energy, including the value of the degree.
Lemma \ref{lemma:covering} finally guarantees the existence of such a
location in sufficiently large domains by a standard covering argument
using the Hardy-Littlewood maximal operator.

Lastly, to prove our concentration result in Theorem
\ref{thm:asymptotics} we use a characterization of weak limits of
minimizing sequences of the Dirichlet energy due to Lin \cite{lin99}:
Such limits are harmonic maps which are smooth away from finitely many
points at which a suitable defect measure indicates concentration of
Dirichlet energy. Due to our boundary data, the harmonic map component
is constant, leaving concentration as the only non-trivial effect.

\section{Auxiliary results}
\label{sec:auxiliary-results}

We start by formulating several key technical results used in the
proof of Theorem \ref{thm:existence}. We recall that the quantity
$\alpha(Q, \kappa)$ appearing in all the lemmas below was defined in
\eqref{def:alpha}.

At the core of the proof of existence of minimizers of $\E$ are simple
lower bounds for the energy that control the $L^2$-norm of $\nabla \m$
for sufficiently small $\kappa$, together with a construction showing
that the infimum energy is strictly below the topological lower bound

\begin{align}
  \label{eq:tlb}
  \int_\Omega |\nabla m|^2 \intd x \geq 8 \pi d \qquad \forall m \in
  \mathcal A_d,
\end{align}
for the case of the pure Dirichlet energy and $d \in \mathbb N$ (see,
e.g., \cite[(3.3)]{melcher14} or \cite[Lemma A.3]{bms:arma21}).

\begin{lemma}\label{lemma:a_priori}
  Let $\kappa > 0$ and $Q\ge 1$.
   For $\m \in H^1(\Omega; \Sphere^2)$ satisfying
  $m = -e_3$ on $\partial \Omega$ we have 
 \begin{enumerate}[i)]
 \item the following lower bounds on the energy:
	\begin{align}
	    \label{eq:Eapriori}
          \E(m) &\geq \left(1 - \alpha(Q,\kappa)\right) \int_{\Omega}
                  |\nabla m|^2 \intd x,\\ 
	 \label{eq:Eapriori2}
          \E(m) &\geq \left(1 -2 \alpha(Q,\kappa)\right) \int_{\Omega}
                  |\nabla m|^2 \intd x + \frac{1}{2} \left(Q-1\right)
                  \int_{\Omega} \left(1-m_3^2\right) \intd x, 
	\end{align}
      \item the following upper bound on the energy:
 \begin{align}
          \label{eq:Eless8pi}
		\inf_{\mathcal{A}_d} \E < 8\pi d.
	\end{align}
 \end{enumerate}
\end{lemma}

As can be seen from estimates \eqref{eq:Eapriori} and
\eqref{eq:Eapriori2}, the energy $\E$ may be used to control
simultaneously the exchange and the anisotropy energy when
  \begin{align}
    \label{eq:alpha_bound}
    \alpha(Q, \kappa) \leq \frac12.
  \end{align}
We also need some basic properties of the function
$\alpha(Q,\kappa)$.

\begin{lemma}\label{lemma:alpha}
  Let $\kappa >0$ and $Q\geq 1$ satisfy inequality
  \eqref{eq:alpha_bound}.  If $\lambda_0 \geq Q-1$, we have
  \begin{align}\label{eq:lambda}
    {\frac{2\kappa^2}{3\lambda_0} }\leq \alpha(Q,\kappa),
  \end{align}
  while for $\lambda_0 < Q-1$, we have
  \begin{align}
    \frac{2 \kappa^2}{3 (Q-1)} \leq \alpha(Q,\kappa).
  \end{align}
\end{lemma}

\noindent Note that the first estimate \eqref{eq:lambda} is
non-optimal if $Q=1$, as then one would have
$\alpha(Q,\kappa) = \lambda_0^{-\frac{1}{2}} \kappa$, but for our
purposes this estimate is sufficient.
 
We now come to the result that is at the heart of our existence
proof. We begin with a basic observation that the minimal energy
cannot go up by more that $8 \pi$, the Dirichlet energy of the degree
1 harmonic map from $\R^2$ to $\mathbb S^2$, when the value of the
degree in the admissible class is increased by 1. This fact is
completely independent of the parameters of the model and simply
reflects the leading order role of the Dirichlet energy.

\begin{proposition}
  \label{p:8pi}
  Let $\kappa \in \R$, $Q \in \R$, and $d \in \mathbb N$. Then
  $\inf_{\mathcal A_{d+1}} \E \leq \inf_{\mathcal A_d} \E + 8 \pi$.
\end{proposition}

\noindent The conclusion of this proposition reflects a possible
bubbling phenomenon: assuming a minimizer over $\mathcal A_d$ exists,
a minimizing sequence from $\mathcal A_{d+1}$ could converge to that
minimizer everywhere except at one point, around which the profile
approaches a sequence of vanishing BP profiles that disappear in the
limit. In this situation the degree of the minimizing sequence would
not be preserved in the limit, failing to yield existence of a
minimizer over $\mathcal A_{d+1}$. Furthermore, in this case we would
have $\inf_{\mathcal A_{d+1}} \E = \inf_{\mathcal A_d} \E + 8 \pi$.

Lemma \ref{lemma:construction} below, which will allow us to rule out
loss of degree in weak limits of minimizing sequences, expresses the
following deeper result: Given enough control on the energy density,
it is possible to insert a carefully chosen, truncated BP profile in
such a way that the energy increases by \emph{strictly} less than
$8 \pi$. The proof requires a careful construction ensuring that the
energy gain in the DMI term of the inserted BP profile wins out over
the error terms incurred in the insertion procedure.

\begin{lemma}\label{lemma:construction}
  There exists a universal constant $\eps>0$ with the following
  property: Let $\kappa >0$ and $Q \geq 1$ satisfy inequality
  \eqref{eq:alpha_bound}, and let $d \in \N \cup\{0\}$. If
  $m \in \mathcal{A}_{d}$ satisfies
  \begin{align} \label{eq:cond} \frac{1}{\pi r^2} \int_{B_r(x)}
    \left(|\nabla m|^2 + {\max\{\lambda_0,Q-1\}} |m+e_3|^2 \right)
    \intd y \leq \eps \kappa^2
\end{align}
for some $x\in \Omega$ and all $r>0$ such that
$B_r(x) \subset \Omega$, then there exists
$\bar m \in \mathcal{A}_{d+1}$ with
\begin{align} \label{eq:compare} \E(\bar m) < \E(m) + 8 \pi.
\end{align}        
\end{lemma}
      
We observe that the above lemma works under an assumption of smallness
of the energy density in a subdomain of $\Omega$. In the following
lemma we show that this assumption is true in sufficiently large or
sufficiently slender domains.

\begin{lemma}\label{lemma:covering}
  There exists a universal constant $\C>0$ such that for all $\eps>0$,
  $\Omega \subset \R^2$ open, bounded, and Lipschitz, $d \in \N$, and
  $\kappa >0$ and $Q \geq1$ satisfying inequality
  \eqref{eq:alpha_bound} the following holds: Let
 \begin{align}\label{eq:beta}
	\beta(\kappa,Q,d) :=  \begin{cases}
		1 & \text{ if } \lambda_0 \geq Q-1,\\
		\frac{ d\kappa^2}{Q-1} &  \text{ if } \lambda_0 < Q-1.
	\end{cases}
\end{align}
If
$|\Omega| \geq \C\left( \beta(\kappa,Q,d) + 1 \right)
\frac{d}{\eps\kappa^2}$, $m \in \mathcal{A}_d$ and $\E(m)\leq 8\pi d$,
then there exists $x\in \Omega$ such that for all $r>0$ with
$B_r(x) \subset \Omega$, we have
\begin{align}
  \frac{1}{\pi r^2} \int_{B_r(x)} \left(|\nabla m|^2 +
  \max\{\lambda_0,Q-1\} |m+e_3|^2 \right) \intd y \leq \eps \kappa^2. 
\end{align}
\end{lemma}

We now proceed to the proofs of the above lemmas.

\begin{proof}[Proof of Lemma \ref{lemma:a_priori}]
  The construction in part ii) of the statement can be done in the
  same way as in \cite[Lemma 3.2]{mmss:cmp23}, inserting $d$ small,
  truncated BP profiles into the domain.

  In order to prove the statement in part i), we consider
  $\alpha >0$ to be determined and estimate using the Poincar{\'e},
  Cauchy-Schwarz, and Young inequalities
\begin{align}
 \begin{split}
   &\quad  \int_\Omega \left(\alpha |\nabla m|^2 -2\kappa m' \cdot
     \nabla m_{3} + (Q-1)  |m'|^2 \right) \intd x\\ 
   & \geq \alpha \| \nabla m_3\|_2^2 + \left(\alpha \lambda_0 + Q-1
   \right) \| m'\|_2^2 -2 \kappa \|m'\|_2 \| \nabla m_3\|_2 \\ 
   & \geq 2 \left( \sqrt{\alpha \left(\alpha \lambda_0 + Q-1 \right)}
     - \kappa \right) \|m'\|_2 \| \nabla m_3\|.
 \end{split}
\end{align}

The smallest $\alpha>0$ ensuring that the right hand side is
non-negative is given by the positive solution of the quadratic
equation
\begin{align}
  \lambda_0 \alpha^2 + (Q-1)\alpha - \kappa^2 =0,
\end{align}
which is given by $\alpha = \alpha(Q, \kappa)$, where
  $\alpha(Q, \kappa)$ is defined in \eqref{def:alpha}.  This proves
the estimate \eqref{eq:Eapriori}.  For
$\tilde \alpha = 2\alpha(Q, \kappa)$ we furthermore have
\begin{align}
  \lambda_0\tilde \alpha^2 + \frac{Q-1}{2} \tilde \alpha - \kappa^2 > 
  \lambda_0 \alpha^2 + (Q-1) \alpha - \kappa^2 = 0, 
\end{align}
similarly giving estimate \eqref{eq:Eapriori2}.
\end{proof}

\begin{proof}[Proof of Lemma \ref{lemma:alpha}]
  If $\lambda_0 \geq Q-1$, we use definition \eqref{def:alpha} to note
  that
  \begin{align}
    \alpha(Q,\kappa) \geq \frac{ 2\frac{\kappa^2}{\lambda_0}}{\sqrt{1+
    4\frac{\kappa^2}{\lambda_0}}+1}  = \frac12 \left( -1 + \sqrt{1
    + {4 \kappa^2 \over \lambda_0}} \right) ,
  \end{align}
  which implies that from $\alpha(Q,\kappa) \leq \frac{1}{2}$ it
  follows that $\frac{\kappa^2}{\lambda_0} \leq {\frac{3}{4}}$ and
    hence $\frac{2\kappa^2}{3\lambda_0}\leq \alpha(Q,\kappa)$.  If on
  the other hand we have $\lambda_0 \leq Q-1$, we note that the
    same type of algebra as above for $\alpha(Q, \kappa) \leq \frac{1}{2}$
    gives ${\kappa^2 \over Q-1} \leq \frac{3}{4}$ and
    ${2 \kappa^2 \over 3(Q - 1)} \leq \alpha(Q, \kappa)$, concluding
  the proof.
\end{proof}

We defer the proof of Proposition \ref{p:8pi} to the end of this
section, as it is independent of the proofs of the remaining lemmas,
and we can also take advantage of the construction in the proof of
Lemma \ref{lemma:construction}.

\begin{proof}[Proof of Lemma \ref{lemma:construction}]
  The strategy is to insert a truncated BP profile in some ball
  $B_\delta(x)$ in order to increase the degree while controlling the
  energy.

  \textit{Step 1. Construction of a cutoff.} Let
  $m \in \mathcal{A}_{d}$ satisfy \eqref{eq:cond}. Shifting domain if
  necessary, we may assume $x=0$, so that for all $\delta>0$ with
  $B_\delta \subset \Omega$, we have
\begin{align}\label{eq:ineqb}
  \int_{B_\delta}  \left( |\nabla \m|^2 + {\max\{\lambda_0,Q-1\} 
  } |m+e_3|^2 \right)\intd x \leq C \eps \kappa^2 \delta^2.
\end{align} 
By \cite[Theorems 4.19 and 4.21]{evans}, up to a redefinition on a set
of measure zero we have
$m |_{\partial B_r} \in H^1(\partial B_r;\mathbb{S}^2)$ with
$\nabla_\tau (m |_{\partial B_r}) = (\nabla_\tau m)|_{\partial B_r}$,
where $\nabla_\tau$ denotes the tangential derivative on the circle,
for almost all $0<r\leq\delta$. From estimate \eqref{eq:ineqb} and an
averaging argument we can therefore conclude that there exists a
circle of radius $\frac{3}{4}\delta \leq r_0 \leq \delta$ such that
\begin{align} \label{eq:ineqc} \int_{\partial B_{r_0}} \left(
    |\nabla_\tau \m |^2 + {\max\{\lambda_0,Q-1\} } |m+e_3|^2 \right)
  \intd \mathcal H^1 \leq C \eps \kappa^2 \delta.
\end{align} 
In particular, we immediately obtain
\begin{align}\label{eq:estimate_tangential}
  \int_{\partial B_{r_0}}\left(  |\nabla_\tau \m'|^2 + 
  {\max\{\lambda_0,Q-1\}  } |\m'|^2\right)  \intd \mathcal H^1
  \leq C \eps \kappa^2 \delta. 
\end{align}

Let now
$(\m')_{r_0} := \frac{1}{2\pi r_0} \int_{\partial B_{r_0}} m' \intd
\mathcal H^1$. Up to a redefinition on a set of $\mathcal H^1$ measure
zero, we have $m'|_{\partial B_{r_0}} \in C(\partial B_{r_0}; \R^2)$,
and for all $x,y \in \partial B_{r_0}$ we have by the fundamental
theorem of calculus and the Cauchy-Schwarz inequality:
\begin{align}
  | m'(x) - m'(y)|^2 \leq C |x-y| \int_{\partial B_{r_0}}
  |\nabla_\tau m'|^2 \intd \mathcal{H}^1. 
\end{align}
Choosing $y$ such that $m'(y) = (m')_{r_0}$ and using estimate
\eqref{eq:estimate_tangential}, for all
$x\in \partial B_{r_0}$ we obtain
\begin{align}\label{est:m0'}
|\m'(x) - (\m')_{r_0}|^2   \leq  C \eps \kappa^2  \delta^2,
\end{align}
from which for all $x\in \partial B_{r_0}$ we get by
  Jensen's inequality and estimate
\eqref{eq:estimate_tangential} that
\begin{align}
  |\m'(x)|^2   \leq  C \eps \kappa^2  \left(
  {\frac{1}{\max\{\lambda_0,Q-1\} }} +    \delta^2    \right). 
\end{align}
Choosing
\begin{align}
  \label{eq:delta}
  \delta \leq \frac{1}{\sqrt{\max\{\lambda_0,Q-1\} }}
\end{align}
and using inequality \eqref{eq:alpha_bound} and Lemma
\ref{lemma:alpha}, for all $x\in \partial B_{r_0}$ we thus have
\begin{align}\label{eq:inplane_small}
  |\m'(x)|^2   \leq   {\frac{C \eps \kappa^2 }{\max\{\lambda_0,Q-1\}
  }}  \leq C \eps. 
\end{align}

Now we want to show that $m_3$ is close to $-1$ on
$\partial B_{r_0}$.  Using the estimate \eqref{eq:ineqc}, we
have
\begin{align}
  \int_{\partial B_{r_0}}\left(  |\nabla_\tau (\m_3 +1)|^2 +
  {\max\{\lambda_0,Q-1\}  
  }  (1+m_3)^2  \right) \intd \mathcal H^1   \leq C \eps \kappa^2 \delta  .
\end{align} 
As in the preceding argument, for all
$ x\in \partial B_{r_0}$ it follows that
\begin{align}
|1+m_3(x)|^2   \leq C \eps.
\end{align} 
Therefore, assuming $\eps$ to be sufficiently small universal and
using $1+m_3 = \frac{|\m'|^2}{1-m_3}$, for all
$x\in \partial B_{r_0}$ we obtain
\begin{align}
|1+m_3(x)|   \leq  C \eps.
\end{align}

Inside $B_{r_0}$ we only have integral estimates on $\m$, but we want
to ensure that $|\m'|$ is small and $m_3$ is close to $-1$
pointwise. In order to achieve this, we will extend $\m'$ from
$\partial B_{r_0}$ inside $B_{r_0}$ in a controlled way, keeping it
small pointwise and not appreciably increasing its Dirichlet energy.
We then recover $m_3$ from the length constraint.

For
$x\in D_{r_0} := B_{r_0}\setminus
B_{r_0/2}$, let
\begin{align}
  v(x) := \left(\frac{2}{r_0}|x| - 1\right) \left(m'\left(r_0
  \frac{x}{|x|} \right) -(\m')_{r_0}\right). 
\end{align}
An explicit calculation together with estimate
\eqref{est:m0'} and assumption \eqref{eq:delta} yields
\begin{align}
  \label{eq:v2}
  \int_{D_{r_0}} |v|^2 \intd x \leq C r_0 \int_{\partial B_{r_0}}
  \left| m' -(\m')_{r_0} \right|^2 \intd \mathcal
  H^1 \leq \frac{C \eps
  \kappa^2\delta^2}{\max\{\lambda_0,Q-1\}}. 
\end{align}
Similarly, with the help of estimate \eqref{eq:estimate_tangential} a
decomposition into tangential and normal derivatives together with the
Poincar{\'e} inequality on $\partial B_{r_0}$ gives
\begin{align}
  \label{eq:dv2}
  \begin{split}
    \int_{D_{r_0}} |\nabla v|^2 \intd x & \leq C r_0 \int_{\partial
      B_{r_0}} \left( \left|\nabla_\tau m' \right|^2 +
      \frac{1}{r_0^2}\left| m' -(\m')_{r_0} \right|^2 \right)\intd
    \mathcal H^1 \\
    & \leq C r_0 \int_{\partial B_{r_0}} \left|\nabla_\tau m'
    \right|^2 \intd \mathcal H^1 \leq C \eps \kappa^2 \delta^2.
  \end{split}
\end{align}

We may now define $\tilde m'(x) := v(x) + (\m')_{r_0}$, which by the
first part of estimate \eqref{eq:inplane_small}, as well as estimates
\eqref{eq:v2} and \eqref{eq:dv2} satisfies
\begin{align}
  \int_{D_{r_0}} |\tilde  m'|^2 \intd x
  & \leq  \frac{C \eps \kappa^2
    \delta^2}{\max\{\lambda_0,Q-1\}},  \label{eq:annulus_anis}\\ 
  \int_{D_{r_0}} |\nabla \tilde m'|^2 \intd x
  & \leq C \eps \kappa^2 \delta^2. \label{eq:annulus_dir}
\end{align}
Using the second part of estimate \eqref{eq:inplane_small} and the
definition of $v$, we have $|v(x)|^2 \leq C \eps$ for any
$x\in D_{r_0}$, as well as $|(\m')_{r_0}|^2 \leq C \eps$.  Therefore,
it follows for all $x\in D_{r_0}$ that
\begin{align}\label{eq:estall_1}
	|\tilde \m'(x)|^2 \leq C  \eps.
\end{align}
Defining $\tilde m_3 = -\sqrt{1-|\tilde \m'|^2}$, for all
$x\in D_{r_0}$ we also get that
\begin{align}\label{eq:estall_2}
	|1+ \tilde m_3(x)| \leq C  \eps.
\end{align}

\textit{Step 2: Energy estimates for the cutoff.} We now want to
estimate the energy terms inside
$D_{r_0} = B_{r_0} \setminus B_{r_0/2}$.  By the estimates
\eqref{eq:annulus_dir} and \eqref{eq:estall_1}, as well as
$|\nabla \tilde m_3(x)|^2 \leq \frac{|\tilde \m'(x)|^2 |\nabla \tilde
  \m'(x)|^2}{1-|\tilde \m'(x)|^2}$ for a.e. $x \in D_{r_0}$, we obtain
for all $\eps >0$ universally small that
\begin{align}
  \label{eq:dm3}
  \int_{D_{r_0}} |\nabla \tilde m_3|^2 \intd x \leq 2 \int_{D_{r_0}}
  |\nabla \tilde \m'|^2 |\tilde \m'|^2 \intd x \leq  C \eps^2 \kappa^2
  \delta^2. 
\end{align}
Therefore, again by estimate \eqref{eq:annulus_dir} we have for
$\eps < 1$:
\begin{align}
  \label{eq:ineqc0}
  \int_{ D_{r_0}} |\nabla \tilde\m |^2
  \intd x \leq C \eps \kappa^2 \delta^2.
\end{align}

The anisotropy term is already controlled by estimate
\eqref{eq:annulus_anis}.  In order to estimate the DMI term, we note
that by the Cauchy-Schwarz inequality, together with estimates
\eqref{eq:estall_1} and \eqref{eq:dm3} we have
\begin{align}\label{eq:ineqc1}
 \int_{D_{r_0}} |\tilde \m' \cdot \nabla \tilde m_3 |
  \intd x \leq 
  \left( \int_{D_{r_0}} |\nabla \tilde m_3|^2 \intd x
  \right)^{1/2}\left( \int_{D_{r_0}}|\tilde \m'|^2 \intd x
  \right)^{1/2} \leq C \eps^{\frac32} \kappa \delta^2. 
\end{align}
\vskip 0.2cm

\textit{Step 3: Inserting a Belavin-Polyakov profile.}  Into the hole
$B_{r_0/2}$, we aim to insert a suitably truncated, rotated
Belavin-Polyakov profile $\phi$ of radius $\rho$ satisfying
$\phi(x) = -e_3$ for $x \in \partial B_{r_0/2}$.
	
Let us recall the construction of an appropriate Belavin-Polyakov
profile \cite{bms:arma21}. To this end, for $L \geq 2$ and $r >0$
we truncate the function
\begin{align}\label{def:f}
  f(r) := \frac{2r}{1+r^2}
\end{align}
via defining
\begin{align}
  f_L(r) := \begin{cases}
    f(r) & \text{ if } r \leq \frac{L}{2},\\
    2 f\left(\frac{L}{2}\right) \left(1 - L^{-1} r \right)& \text{ if
    } \frac{L}{2} < r \leq L, \\ 
    0 & \text{ if } L < r.
  \end{cases}
\end{align}
For $x \in \R^2$ this translates to a truncated Belavin-Polyakov
profile
\begin{align}
  \Phi_L(x) := \left(-f_L(|x|) \frac{x}{|x|},
  \operatorname{sign}(1-|x|) \sqrt{1- f^2_L(|x|)} \right), 
\end{align}
with the original Belavin-Polyakov profile being
\begin{align}
  \Phi(x) := \left( -\frac{2 x}{ 1 + |x|^2},
  \frac{1-|x|^2}{1+|x|^2} \right)  =  \lim_{L \to \infty} \Phi_L(x). 
\end{align}
	
In the spirit of estimates \cite[(A.66) to (A.71)]{bms:arma21}, we
obtain
\begin{align}\label{eq:bp_trunc_exch}
  \begin{split}
    \int_{\R^2} |\nabla \Phi_L|^2 \intd x & \leq \frac{8 \pi L^2}
    {4+L^2}
    + 2\pi \int_{\frac{L}{2}}^{L} r \left( \frac{4
        f^2\left(\frac{L}{2}\right)L^{-2}}{1-
        f^2\left(\frac{L}{2}\right) } + \frac{f^2_L(r)}{r^2}\right)
    \intd r \\ 
    & \leq 8\pi + CL^{-2},
  \end{split}
\end{align}
where $C>0$ is some universal constant.  In addition, we have
$\mathcal{N}(\Phi_L) =1$ for all $L \geq 2$, since
$\nabla \Phi_L \to \nabla \Phi$ in $L^2(\R^2)$ and $\Phi_L \to \Phi$
pointwise a.e.\ as $L \to \infty$, and $\mathcal{N}(\Phi_L)$ is a
  continuous function of $L \geq 2$ with values in $\mathbb N$.
	
Now let us modify this truncated profile by a suitable rotation to
match the boundary conditions
$ \phi_\rho=\left( (\m')_{r_0}, - \sqrt{1 - (\m')_{r_0}^2}\right)$ on
$\partial B_{r_0/2}$.  Let $R \in SO(3)$ be such that
\begin{align}
  -Re_3 = \left( (\m')_{r_0}, - \sqrt{1 - (\m')_{r_0}^2}\right),
\end{align}
which due to the estimate \eqref{eq:inplane_small} for $\m'$ we may
choose to satisfy
\begin{align}
  \label{eq:Rest}
  |R - \operatorname{id} |^2 \leq  {\frac{C  \eps \kappa^2
  }{\max\{\lambda_0,Q-1\} }} \leq C\eps ,
\end{align}
where $| \cdot |$ denotes the usual Frobenius norm.  Then for
$0< \rho \ll \delta$ and
$\phi_\rho(x) := R \Phi_{\frac{r_0}{2\rho}}(\rho^{-1}x)$ we have
$\mathcal{N}(\phi_\rho) = 1$ and
$\phi_\rho (x) = \left( (\m')_{r_0}, - \sqrt{1 -
    (\m')_{r_0}^2}\right)$ for all $x\in \partial B_{r_0/2}$.  Due to
the estimate \eqref{eq:bp_trunc_exch}, the rotation $R$ not affecting
the Dirichlet energy and recalling that
$\frac{3}{4}\delta \leq r_0 \leq \delta$, we also get
\begin{align}\label{eq:phi_exchange}
  \int_{\R^2} | \nabla \phi_\rho|^2 \intd x
  & \leq 8 \pi + \frac{C \rho^2}{\delta^2}. 
\end{align}

To estimate the DMI energy of $\phi_\rho$, we use the argument in the
proof of estimate \cite[(A.53)]{bms:arma21} to obtain
\begin{align}
  \int_{\R^2} \Phi'_L \cdot \nabla \Phi_{L,3} \intd x >
  \frac{1}{C}, 
\end{align}
for all $L$ sufficiently large universal.  Since $\Phi_L(x)+e_3$ and
$\Phi(x) +e_3$ decay as $\frac{1}{|x|}$ and $\nabla \Phi_L(x)$ and
$\nabla \Phi(x)$ decay as $\frac{1}{|x|^2}$ as $|x| \to \infty$, with
the help of estimate \eqref{eq:Rest} for $L$ sufficiently large and
$ \eps$ sufficiently small universal we also observe that
\begin{align}
  \int_{\R^2} (R \Phi_L)' \cdot \nabla (R \Phi_{L})_3
  \intd x > \frac{1}{C}. 
\end{align}
In total, this gives
\begin{align}
  \label{eq:phiDMIest}
  \int_{B_{\frac{r_0}{2}}} \phi_\rho' \cdot \nabla \phi_{\rho,3}
  \intd x & \geq \frac{\rho}{C},
\end{align}
for all $\eps$ and $\rho / \delta$ sufficiently small universal.

Additionally, due to
$|\phi_\rho'(x)|^2\leq C \left(\frac{\rho^2 |x|^2}{(\rho^2+ |x|^2)^2}
  +|R -\operatorname{id} |^2 \right)$ for $x\in B_{r_0/2}$ and
estimate \eqref{eq:Rest} we have
\begin{align}
  \label{eq:phi'est}
  \int_{B_{\frac{r_0}{2}}} |\phi'_\rho|^2 \intd x  \leq C\left(
  \rho^2 \log \frac{\delta}{\rho} + \  {\frac{\eps
  \kappa^2 \delta^2}{\max\{\lambda_0,Q-1\} }  } \right),
\end{align}
again for all $\eps$ and $\rho / \delta$ sufficiently small universal.

Putting together estimates \eqref{eq:phi_exchange},
  \eqref{eq:phiDMIest} and \eqref{eq:phi'est}, we arrive at
\begin{align}\label{eq:energy_phi}
  \E (\phi_\rho) \leq 8\pi +  C \left( \frac{\rho^2}{\delta^2} + (Q-1)
  \rho^2 \log \frac{\delta}{\rho}  \right) - \frac{\kappa \rho}{C} + C
  \eps \kappa^2 \delta^2. 
\end{align}

\textit{Step 4. Construction of a competitor.}  We now construct a
degree $d+1$ competitor $\bar\m$ by inserting $\phi_\rho$ into
$B_{r_0/2}$ via
\begin{align}
  \bar \m(x) = \begin{cases}
    \phi_{\rho} (x)& |x| \leq \frac{r_0}{2}, \\
    \tilde \m \left( x \right)  & \frac{r_0}{2}<|x|<r_0,\\
    \m(x) &   |x| \geq r_0,
\end{cases}
\end{align}
for $x\in \Omega$.  As on $\partial B_{r_0/2}$ we have
\begin{align}
	\phi_\rho=\left( (\m')_{r_0}, - \sqrt{1 - (\m')_{r_0}^2}\right)  = \tilde m,
\end{align}
the map $\bar \m$ is well defined in $H^1(\Omega; \mathbb S^2)$.

\textit{Step 5: Prove $\bar m \in \mathcal{A}_{d+1}$.} It is clear due
to construction that we have
\begin{align}
  \mathcal N (\bar \m) |_{B_{r_0/2}} := \frac{1}{4
  \pi}\int_{B_{r_0/2}}  \bar \m \cdot (\partial_1 \bar \m 
  \times \partial_2 \bar \m) \intd x =1.
\end{align}
Additionally, by estimates \eqref{eq:ineqb} and \eqref{eq:ineqc0} we
have
\begin{align}
\begin{split}
  \left|\mathcal N (\bar \m) - \mathcal N (\m) - \mathcal N (\bar \m)
    |_{B_{r_0/2}}\right| &= \left| \int_{D_{r_0}} \bar \m \cdot
    (\partial_1 \bar \m \times \partial_2 \bar \m) \intd x -
    \int_{B_{r_0}} \m \cdot (\partial_1 \m
    \times  \partial_2 \m) \intd x\right|  \\
  &\leq C \left( \int_{D_{r_0}} |\nabla \tilde \m|^2 \intd x
    +
    \int_{B_{r_0}} |\nabla \m|^2 \intd x \right) \\
  & \leq C \eps \kappa^2 \delta^2.
   \end{split}
  \end{align}
  As we know that $ \mathcal N (\m) =d$, in view of discreteness of
  the degree and the assumption \eqref{eq:alpha_bound} we deduce that
  $\mathcal N (\bar\m)=d+1$ for any choice of $\delta$ satisfying
  \eqref{eq:delta}, with $\eps >0$ small universal.

  \textit{Step 6: Conclusion.} We know that outside of $B_{r_0}$
  the maps $\m$ and $\bar\m$ coincide. Therefore, we just need to show
  that when restricted to $B_{r_0}$ we have
\begin{align}
\E  (\bar \m) |_{B_{r_0}} < \E (\m) |_{B_{r_0}} + 8 \pi.
\end{align}
Using estimates \eqref{eq:ineqb}, \eqref{eq:alpha_bound} and Lemma
\ref{lemma:alpha}, we know that inside $B_{r_0}$ we have
$\int_{B_{r_0}} |\m'|^2 \intd x \leq C \eps \delta^2$ and hence by the
Cauchy-Schwarz and Young inequalities
\begin{align}\label{eq:inside}
  \E (\m) |_{B_{r_0}} \geq \int_{B_{r_0}} |\nabla \m|^2 \intd x - 2
  \kappa  \int_{B_{r_0}} \nabla m_3 \cdot \m'  \intd x \geq -C
  \kappa^2  \int_{B_{r_0}} |\m'|^2  \intd x \geq -C  \eps \kappa^2
  \delta^2. 
\end{align}
Moreover, using estimates \eqref{eq:annulus_anis}, \eqref{eq:ineqc0},
\eqref{eq:ineqc1}, and \eqref{eq:energy_phi} we obtain
\begin{align}
\begin{split}
  \E (\bar \m) |_{B_{r_0}} &\leq 8\pi + C \left(
    \frac{\rho^2}{\delta^2} + (Q-1) \rho^2 \log \frac{\delta}{\rho}
  \right) - \frac{\kappa \rho}{C} + C \eps \kappa^2 \delta^2.
 \end{split}
\end{align}
Choosing $\rho= \kappa \delta^2/(2C^2)$, we arrive at
\begin{align}
  \E |_{B_{r_0}} (\bar \m)  \leq 8 \pi - \frac{\kappa^2
  \delta^2}{4C^3} + C(Q-1) \kappa^2 \delta^4 \ln \left(
  \frac{1}{\kappa \delta}\right) + C \eps \kappa^2 \delta^2. 
\end{align}
Taking $\eps>0$ small enough universal, $\delta>0$ small enough
depending on $\lambda_0$, $Q$ and $\kappa$ (it is enough to ensure
that in addition to \eqref{eq:delta} the quantities
  $\kappa \delta$, $(Q-1) \delta$ and $\delta \ln(1/(\kappa\delta))$
  are all sufficiently small universal), we arrive at
\begin{align}
  \E |_{B_{r_0}} (\bar \m) -\E (\m) |_{B_{r_0}} \leq 8 \pi -
  \frac{\kappa^2 \delta^2}{5C^3} + C\eps \kappa^2 \delta^2  <8\pi, 
\end{align}
which concludes the proof.
\end{proof}

\vskip 0.5cm

\begin{proof}[Proof of Lemma~\ref{lemma:covering}]
  Under our assumptions on $\kappa$, $Q$, and the energy, estimate
    \eqref{eq:Eapriori} of Lemma \ref{lemma:a_priori} implies
  that
\begin{align}\label{eq:upper_dirichelt}
	 \int_{\Omega} |\nabla m|^2 \intd x \leq 16 \pi d,
\end{align}
so that by the Poincar{\'e} inequality, we have
\begin{align}
	\lambda_0 \int_\Omega |m+e_3|^2 \intd x \leq 16 \pi d.
\end{align}
On the other hand, estimate \eqref{eq:Eapriori2} of Lemma
\ref{lemma:a_priori} and the topological lower bound
  \eqref{eq:tlb} give
\begin{align}
  8 \pi d \left(1 -2 \alpha(Q,\kappa)\right) + \frac{1}{2}
  \left(Q-1\right) \int_{\Omega} 
  \left(1-m_3^2\right) \intd x \leq 8 \pi d,
\end{align}
which by \eqref{est:al} then implies
\begin{align}
  \label{eq:aniQm12}
 \int_{\Omega}
  \left(1-m_3^2\right) 
  \intd x \leq  \frac{C d \alpha(Q,\kappa)}{(Q-1)} \leq  \frac{C d
  \kappa^2}{(Q-1)^2} .
\end{align}
Thus, in view of the fact that the extension of $m$ satisfies
  $m = -e_3$ outside $\Omega$, by \cite[Lemma 5.1]{bms:arma21} we
  obtain
\begin{align}\label{eq:mneqe3}
  \int_{\Omega} |m_3+1|^2 \intd x \leq \frac{1 }{ 4\pi} \int_{\Omega}
  |\nabla m|^2 \intd x  \int_\Omega \left(1- m_3^2\right)  \intd x
  \leq \frac{C d^2\kappa^2}{(Q-1)^2}. 
\end{align}
 In total, we get
\begin{align}
  \max\{\lambda_0,Q-1\}  \int_{\Omega} |m_3+1|^2 \intd x \leq C
  \beta(\kappa,Q,d)  d, 
\end{align}
where we recall that
\begin{align}
  \beta(\kappa,Q,d) = \begin{cases}
    1 & \text{ if } \lambda_0 \geq Q-1,\\
    \frac{ d\kappa^2}{Q-1} &  \text{ if } \lambda_0 < Q-1.
  \end{cases}
\end{align}

For $x \in \R^2$ we consider the Hardy-Littlewood maximal functions
\begin{align}
  M_1(x) & := \sup_{r>0} \frac{1}{\pi r^2} \int_{B_r(x)} |m + e_3|^2
           \intd y,\\ 
  M_2(x) & := \sup_{r>0} \frac{1}{\pi r^2} \int_{B_r(x)} |\nabla m |^2
           \intd y, 
\end{align}
which are well known to be bounded in weak $L^1$ if the original
functions are bounded in $L^1$, see \cite[Chapter 3, Theorem
1.1]{stein-shakarchi}.  Consequently, for all $t >0$, there exists a
universal constant $C_1>0$ such that
\begin{align}
  \left| \left\{ M_1(x) > t  \right\} \right|
  & \leq \frac{C_1 \beta(\kappa,
    Q,d) d }{ t
    \max \{\lambda_0, Q -
    1\}  },\\ 
  \left| \left\{ M_2(x) > t  \right\} \right|
  & \leq \frac{C_1  d }{ t}. 
\end{align}
Choosing
$t_1:= 3C_1 \frac{\beta(\kappa, Q,d) d}{|\Omega| \max \{\lambda_0, Q -
  1\} }$ and $t_2 := 3C_1 \frac{d}{|\Omega|}$ we get that
\begin{align}
  \left| \Omega \cap  \left\{ M_1(x) \leq  t_1  \right\} \right|
  & \geq \frac{2}{3} |\Omega|,\\
  \left| \Omega \cap  \left\{ M_2(x) \leq  t_2  \right\} \right|
  & \geq \frac{2}{3} |\Omega|.
\end{align}

Thus, there exists $x\in \Omega$ such that
\begin{align}
  M_1(x) & \leq 3C_1\frac{  \beta(\kappa, Q,d) d}{|\Omega| \max
           \{\lambda_0, Q - 1\} }
\end{align}
and
\begin{align}
  M_2(x) & \leq 3C_1 \frac{ d}{|\Omega|}.
\end{align}
The statement follows provided
$3C_1 \frac {\beta(\kappa, Q,d) d}{|\Omega|} \leq \eps \kappa^2$ and
$3C_1\frac{d}{|\Omega|}\leq \eps \kappa^2$, which is the case under
our assumption on $|\Omega|$ for $C_0=3 C_1$.
\end{proof}

\begin{proof}[Proof of Proposition \ref{p:8pi}]
  Letting $m \in \mathcal A_d$ and passing to the precise
  representative if necessary \cite[Theorem 4.19]{evans}, pick
  $x \in \Omega$ to be a point of continuity of $m$ that is also a
  Lebesgue point of $\nabla m$. Without loss of generality, we may
  assume that $x = 0$ and, hence,
  \begin{align}
    \label{eq:mdel2m0}
    \int_{B_\delta} |\nabla m|^2 \intd x \leq C \delta^2, \qquad
    m(0) = \lim_{|y| \to 0} m(y), \qquad |m(0)| = 1.
  \end{align}
  for some $C > 0$ independent of $\delta \ll 1$. 

  Up to a rigid rotation of $m$, we may assume for the moment that
  $m(0) = -e_3$.  Arguing as in the proof of Lemma
  \ref{lemma:construction}, we can find
  $r_0 \in \big( \frac34 \delta, \delta \big)$ such that
  $m |_{\partial B_{r_0}} \in H^1(\partial B_{r_0}; \mathbb S^2)$ and
  $|m' - (m')_{r_0}| \leq C \delta$ for some constant $C > 0$
  independent of $\delta \ll 1$. Furthermore, by continuity of $m$ at
  the origin we have $m_3(x) < 0$ for all $x \in \partial B_{r_0}$ and
  $\delta \ll 1$. Therefore, as in the proof of Lemma
  \ref{lemma:construction} we can define a cutoff $\tilde m_\delta$ in
  $D_{r_0}$ for every $\delta \ll 1$ such that $\tilde m_\delta = m$
  on $\partial B_{r_0}$,
  $\tilde m_\delta = \left( (m')_{r_0}, -\sqrt{1 - (m')_{r_0}^2}
  \right)$ on $\partial B_{r_0/2}$, and
  \begin{align}
    \label{eq:dm2del2}
    \int_{D_{r_0}} |\nabla \tilde m_\delta|^2 \intd x \leq C \delta^2,
  \end{align}
  for some $C > 0$ independent of $\delta \ll 1$. By the
  Cauchy-Schwarz inequality, we also have
  \begin{align}
    \label{eq:dmidel2}
    \int_{D_{r_0}} |\tilde m'_\delta \cdot \nabla \tilde m_{\delta,3}|
    \intd x \leq C \delta^2, 
  \end{align}
  for some $C > 0$ independent of $\delta \ll 1$. As the estimates in
  \eqref{eq:dm2del2} and \eqref{eq:dmidel2} are rotation-invariant,
  existence of an $\tilde m_\delta$ satisfying these estimates and
  interpolating from $m$ on $\partial B_{r_0}$ to a constant unit
  vector on $\partial B_{r_0/2}$ follows for an arbitrary value of
  $m(0)$.

  We now set $\rho_\delta = \delta^2$, and for $\delta \ll 1$ and
  $x \in B_{r_0/2}$ we define a truncated Belavin-Polyakov profile
  $\phi_{\rho_\delta}(x) = R \Phi_{\rho \over 2 r_0}(\rho_\delta^{-1}
  x)$, where $R \in SO(3)$ is a rotation satisfying
  $R e_3 = -\tilde m_\delta|_{\partial B_{r_0/2}}$. By an explicit
  calculation we get
  \begin{align}
    \label{eq:dphi2dmi}
    \int_{B_{r_0/2}} |\nabla \phi_{\rho_\delta}|^2 \intd x \leq 8 \pi + C
    \delta^2, \qquad \int_{B_{r_0/2}} |\phi_{\rho_\delta} ' \cdot \nabla
    \phi_{\rho_\delta,3}| \intd x \leq C \delta^2,
  \end{align}
  for some $C > 0$ independent of $\delta \ll 1$. 
  
  Finally, we construct a competitor
  $\bar m_\delta \in H^1(\Omega; \mathbb S^2)$ exactly as in Lemma
  \ref{lemma:construction}. Clearly, by
  \eqref{eq:mdel2m0},\eqref{eq:dm2del2} and the definition of
  $\phi_{\rho_\delta}$ we have $\bar m_\delta \in \mathcal A_{d+1}$
  for all $\delta \ll 1$. Furthermore, by
  \eqref{eq:mdel2m0}--\eqref{eq:dphi2dmi} and again the Cauchy-Schwarz
  inequality for the DMI term in $\E(m)$ we have
  \begin{align}
    \label{eq:EE8piCd2}
    \E(\bar m_\delta) \leq \E(m) + 8 \pi + C \delta^2,
  \end{align}
  for some $C > 0$ independent of $\delta \ll 1$. We then conclude by
  applying \eqref{eq:EE8piCd2} to a minimizing sequence
  $(m_n) \in \mathcal A_d$ and using a diagonal argument.
\end{proof}

\begin{remark}
  We note that contrary to Lemma \ref{lemma:construction}, the
  construction in the proof of Proposition \ref{p:8pi} does not allow
  one to conclude a strict inequality, since the rotation $R$ is a
  priori not close to identity and, hence, we may not pick a negative
  contribution from the DMI term. Furthermore, even if we were able to
  show that the competitor produces a net negative contribution to the
  DMI energy, we would not be able to conclude a priori that it beats
  the possible positive gain in the exchange energy due to the cutoff.
\end{remark}

\section{Proofs of the main theorems}
\label{s:proofs}

Now we conclude the proofs of Theorem \ref{thm:existence}, Proposition
\ref{p:reg}, and Theorem \ref{thm:asymptotics}.

\begin{proof}[Proof of Theorem \ref{thm:existence}]
  \textit{Step 1.  Ensuring the applicability of Lemma
    \ref{lemma:covering}.}  We first show that under assumption
  \eqref{eq:smallness_kappa_d} it is possible to find
  $\overline{C} > 0$ universal such that if assumption
  \eqref{eq:cond_omega_big} holds, then the condition on $|\Omega|$ of
  Lemma~\ref{lemma:covering} with $\eps > 0$ from Lemma
  \ref{lemma:construction} is satisfied.  That is, we need to ensure
  that
  $|\Omega| \geq C_0\left( \beta(\kappa,Q,d) + 1 \right)
  \frac{d}{\eps\kappa^2}$ with $\eps>0$ universal from Lemma
  \ref{lemma:construction}, $\beta(\kappa,Q,d)$ defined in
  \eqref{eq:beta}, and $C_0 > 0$ being a universal constant from
  Lemma~\ref{lemma:covering}.

  If $\lambda_0 \geq Q-1$, then $\beta(\kappa, Q, d) = 1$, and
  using inequality \eqref{eq:cond_omega_big} we have
\begin{align}
  |\Omega| \geq \frac{\overline{C}  d}{\kappa^2} = C_0 \left(
  \beta(\kappa,Q,d) +  1 \right) 
  \frac{d}{\eps\kappa^2},
\end{align}
if $\overline{C} = {2 C_0/\eps}$.  Otherwise, we have
$\lambda_0 < Q-1$.  Using condition \eqref{eq:smallness_kappa_d} and
Lemma \ref{lemma:alpha}, we have
\begin{align}
  \frac{2}{d} \geq \alpha(Q,\kappa) \geq \frac{2 \kappa^2}{3(Q-1)}, 
\end{align}
and therefore
$\beta(\kappa, Q, d) = \frac{d \kappa^2}{Q-1} \leq
3$. Consequently, invoking the inequality \eqref{eq:cond_omega_big},
we have
\begin{align}
  |\Omega| \geq \frac{\overline{C}  d}{\kappa^2} \geq C_0
  \left( \beta(\kappa,Q,d)  +1 \right) 
  \frac{d}{\eps \kappa^2},
\end{align}
if $\overline{C}= 4 C_0/\eps$.

\textit{Step 2. Convergence of minimizing sequences and lower
  semicontinuity.}  Let $(\m_n) \in \mathcal{A}_d$ be a minimizing
sequence.  By Lemma \ref{lemma:a_priori}, since inequality
\eqref{eq:alpha_bound} holds by assumption, we get that $(\m_n)$ is
uniformly bounded in $H^1(\Omega;\mathbb{S}^2)$.  Consequently, there
exists a subsequence (not relabeled) and
$\m_\infty \in H^1(\Omega; \mathbb{S}^2)$ such that
$\m_n \to \m_\infty $ in $L^2$ and
$\nabla \m_n \rightharpoonup \nabla \m_\infty$ in $L^2$ as
$n \to \infty$.  Furthermore, by a weak-times-strong argument, we get
$\int_{\Omega} \m_n' \cdot \nabla m_{n,3} \intd x \to \int_{\Omega}
\m_\infty' \cdot \nabla m_{\infty,3} \intd x$ and, therefore, we have
\begin{align}
  \E(\m_\infty) \leq \liminf_{n\to \infty} 
  \E (\m_n) = \inf_{\mathcal A_d} \E. 
\end{align}
Thus, it remains to prove that $m_\infty \in \mathcal A_d$, i.e., that
$\mathcal N(m_\infty) = d$.
        
Arguing as in \cite{brezis83a,melcher14} and \cite[Lemma
A.3]{bms:arma21}, we complete the squares to get for all
$\m \in H^1(\Omega; \mathbb{S}^2)$ that
\begin{align}
  \label{eq:squares}
  \int_\Omega |\nabla \m|^2 \intd x  \pm 8\pi \mathcal N (\m) 
  &= \int_\Omega |\partial_1 \m \mp \m \times \partial_2 \m|^2 \intd x.
\end{align}
As a result, by the lower semicontinuity of the right-hand side in
\eqref{eq:squares} and the continuity of the DMI and anisotropy terms
we have
\begin{align}
  \label{eq:ineq1}
  \E(\m_\infty) \pm 8\pi
  \mathcal{N}(\m_\infty) 
  & \leq 
    \liminf_{n \to \infty} \E(m_n)  \pm 8\pi d. 
\end{align}

\textit{Step 3.  Proving $1 \leq \mathcal{N}(m_\infty) \leq d$.}
With the help of
Lemma~\ref{lemma:a_priori} we know that $\E(\m_\infty) \geq 0$ and
$ \inf_{\mathcal A_d} \E < 8\pi d$. Therefore, using
\eqref{eq:ineq1} with the ``--'' sign'', we obtain
$ - 8\pi \mathcal{N}(\m_\infty) <0$, implying
$\mathcal{N}(\m_\infty) \geq 1$.

The inequality \eqref{eq:ineq1} with the ``+'' sign, together
with Lemma \ref{lemma:a_priori} and the topological bound
\eqref{eq:tlb} yield
\begin{align}
  8\pi (2 - \alpha(Q,\kappa)) \mathcal{N}(m_\infty) < 16\pi d.
\end{align}
Therefore, under assumption \eqref{eq:smallness_kappa_d} we have
\begin{align}
  \mathcal{N}(m_\infty) < \frac{2d}{2-\alpha(Q,\kappa)} \leq d+1.
\end{align}
Thus, by discreteness of the degree we have
$\mathcal N(m_\infty) \leq d$. Note that this is the only place in our
argument where we crucially need $\alpha(Q, \kappa)$ to be bounded
proportionally to $d^{-1}$.

\textit{Step 4.  Proving $\mathcal{N}(m_\infty) = d$.} If $d=1$,
  we are done now, producing a minimizer of $\E$ over $\mathcal A_1$
  (compare with \cite[Theorem 2.1]{mmss:cmp23} in the case $Q =
  1$). So in the following we may assume $d\geq 2$.

We argue by contradiction and assume
$N: = {\mathcal N}(\m_\infty) <d$. Using estimate \eqref{eq:ineq1}, we
have
\begin{align}\label{eq:contr}
  \inf_{\mathcal A_N} \E + 8\pi \left(d -N \right)&\leq \inf_{\mathcal A_d} \E.
\end{align}
Since $1 \leq N <d$ and we already have existence for $d=1$, we can
use induction and assume that the minimum of $\E$ over
  $\mathcal A_{d'}$ is attained for all $1 \leq d' < d$, with
  $\min_{\mathcal A_{d'}} \E < 8 \pi d'$ by Lemma
  \ref{lemma:a_priori}. Then applying first Lemma
  \ref{lemma:covering}, then Lemma~\ref{lemma:construction} repeatedly
  to the minimizers of $\E$ in $\mathcal A_{d'}$ for $N \leq d' < d$,
we obtain
\begin{align}
  \inf_{\mathcal
  A_d} \E <  \min_{\mathcal A_N} \E + 8\pi \left(d - N \right). 
\end{align}
However, this contradicts estimate \eqref{eq:contr}, and, therefore,
the assumption $N < d$ is false, proving the assertion of the theorem.
\end{proof}

  \begin{proof}[Proof of Proposition \ref{p:reg}]
    Arguing as in \cite{schoen82}, let
    $\xi \in C_c^\infty (\Omega; \R^3)$. Then, for $|t| < t_0$ with
    $t_0>0$ small enough, we have
  \begin{align}
    m^t:= \frac{m + t \xi}{|m +t \xi|} \in \mathcal A_d. 
  \end{align}
  Indeed, for $i=1,2$, we get $|m^t| = 1$ a.e. in
    $\Omega$, $m^t = -e_3$ on $\partial \Omega$ and
  \begin{align}
    \label{eq:mt}
    m^t & = m + t \left( \xi - (\xi\cdot m) m \right) + O(t^2),\\
    \partial_i m^t & = \partial_i m + t \left( \partial_i \xi -
                     \left(m \cdot \partial_i \xi + \xi \cdot
                     \partial_i m\right) m  - (\xi \cdot m) \partial_i
                     m \right) + O(t^2 (1 + |\nabla m|)), 
  \end{align}
  for a.e. $x \in \Omega$, where the constants in the $O$-notation
  depend on $\xi$, but not on $m$.  In particular, we have
  $m^t \in H^1(\Omega; \mathbb S^2)$ for all $|t|<t_0$ and $m^t$ is
  continuous in $H^1(\Omega; \R^3)$ at $t = 0$. Therefore, by
  continuity of the degree in $H^1(\Omega; \R^3)$ we have
  $\mathcal N(m^t) = d$ in a sufficiently small neighborhood of
  $t = 0$.

    We now compute the derivative of $\E(m^t)$ at $t = 0$. Rewriting
  the anisotropy term as $(Q-1) (1- m_3^2)$ and interpreting
  $\nabla m_3 = (\partial_1 m_3, \partial_2 m_3 ,0)$, by
    minimality of $m$ we have
  \begin{align}
    \begin{split}
      0 = \left. \frac{\intd }{\intd t} \E (m^t) \right |_{t=0} & =2
      \int_\Omega \left( \nabla m : \nabla \xi - |\nabla m|^2 m
        \cdot \xi \right) \intd x \\
      &\quad + 2 \kappa \int_\Omega \left( (\nabla \cdot m')(e_3 - m_3
        m )-
        (\nabla m_3 - (m \cdot \nabla m_3) m \right) \cdot \xi \intd x \\
      & \quad - 2 (Q-1) \int_\Omega \left( m_3 e_3 - m_3^2 m \right)
      \cdot \xi \intd x.
  \end{split}
  \end{align}
  This gives the distributional version of equation \eqref{eq:EL}.
  
  As the lower order terms in this equation are all $L^2$-integrable,
  standard regularity theory for harmonic maps in $\R^2$
  \cite{helein90,coifman93,chang99} implies that
  $m \in C^\infty(\Omega ; \mathbb{S}^2)$, with continuity up to the
  boundary established in \cite{mueller09}. Specifically, a result by
  M{\"u}ller and Schikorra, \cite[Theorem 1.1, Remark 1.4, and
  Corollary 1.6]{mueller09}, implies that $m$ is locally H{\"o}lder
  continuous for some $\gamma \in (0,1)$ and, if $\Omega$ is
  additionally simply connected with a $C^{1,\alpha}$ boundary, also
  continuous up to the boundary. H\"older continuity of the
  derivatives of $m$ in the interior then follows from the arguments
  of Giaquinta \cite{giaquinta} (for the specifics, see \cite[Section
  3]{chang99}), with smoothness then following from the bootstrap
  argument of the standard elliptic regularity theory \cite{gilbarg}.
\end{proof}

\vskip 0.2cm

\begin{proof}[Proof of Theorem \ref{thm:asymptotics}]
  By inequality \eqref{est:al} the condition
  \eqref{eq:smallness_kappa_d} is satisfied for $Q_n\gg1$ and,
  therefore, existence of a minimizer $\m_n \in \mathcal A_d(\Omega)$
  follows from Theorem \ref{thm:existence} for all $n$ large enough.
	
  By Lemma \ref{lemma:a_priori} and the Poincar{\'e} inequality, the
  sequence $(m_n + e_3)$ is bounded in $W^{1,2}(\R^2;\R^3)$ after
  constant extension by $-e_3$ outside of $\Omega$.  Indeed, using
  Lemma~\ref{lemma:a_priori} we have
  \begin{align}
    \lim_{n \to \infty} \int_{\R^2} |\nabla
    \m_n|^2 \intd x = 8 \pi d, 
  \end{align}
  since $\alpha(Q_n,\kappa) \to 0$ by inequality \eqref{est:al}.  For
  $x\in \R^2$, let
  \begin{align}
    \Phi(x) := \begin{pmatrix}
      -\frac{2x}{1+|x|^2} & \frac{1-|x|^2}{1+|x|^2}
    \end{pmatrix}.
  \end{align}
  Then by conformal invariance of the Dirichlet energy in two
  dimensions, see for example \cite[Lemma A.2]{bms:arma21}, also
  $\hat m_n : \mathbb{S}^2 \to \mathbb{S}^2$ defined as
  $\hat m_n := m_n \circ \, \Phi^{-1}$ satisfies
  \begin{align}
    \label{eq:limharmd}
    \lim_{n \to 0} \int_{\mathbb{S}^2} |\nabla \hat
    m_n|^2 \intd \mathcal{H}^2 = 8\pi d. 
  \end{align}
  Consequently, in view of the topological lower bound \eqref{eq:tlb}
  it is a minimizing sequence of the Dirichlet energy among maps from
  $\mathbb S^2$ to $\mathbb S^2$ of degree $d$.
	
  Additionally, by Proposition \ref{p:reg} we have
  $m_n \in C(\R^2;\mathbb{S}^2)$ and thus
  $\hat m_n \in C(\mathbb{S}^2;\mathbb{S}^2)$.  By \cite[Theorem
  1'']{lin99}, there exists a subsequence, a weakly harmonic map
  $\hat m_\infty \in W^{1,2}(\mathbb{S}^2;\mathbb{S}^2)$ and a
  non-negative Radon measure $\nu$ on $\mathbb{S}^2$ such that
  \begin{align}
    \hat m_n \rightharpoonup \hat m_\infty
  \end{align}
  in $W^{1,2}(\mathbb{S}^2;\mathbb{S}^2)$, and for
  $d \hat\mu_n := |\nabla \hat m_n|^2 \intd \mathcal{H}^2$ and
  $d \hat\mu_\infty := |\nabla \hat m_\infty|^2 \intd \mathcal{H}^2$ we
  have
  \begin{align}
    \hat\mu_n
    \stackrel{*}{\rightharpoonup} \hat\mu_\infty + \nu ,
  \end{align}
  as measures when $n \to \infty$. By \cite[Theorem 5.8]{lin99}, there
  exist $k \in \N \cup \{0\}$, $z_1,\ldots, z_k \in \mathbb{S}^2$, and
  $d_1,\ldots, d_k \in \N$ such that
  \begin{align}
    \nu = \sum_{j=1}^k 8\pi d_j \delta_{z_j},
  \end{align}
  with the convention that $\nu = 0$ if $k = 0$.	
	
  Turning our attention to $\hat m_\infty$, we observe that because it
  is a stationary point of the Dirichlet energy on the sphere, it must
  be an energy-minimizing map in its own homotopy class, a fact that
  was first independently proved by Lemaire \cite{lemaire78} and Wood
  \cite{wood74}, see also \cite[(11.5)]{eells78}.  However, since
  $\hat m_\infty (x) = -e_3$ for
  $x\in \mathbb{S}^2 \setminus \Phi(\Omega)$, we must have that
  $\hat m_\infty = -e_3$ on the entirety of $\mathbb{S}^2$, which
  follows from \cite[(8.4)]{lemaire78}.
	
  We thus have
  \begin{align}
    \label{eq:8piddzj}
    \hat\mu_n
    \stackrel{*}{\rightharpoonup}\sum_{j=1}^k 8\pi d_j
    \delta_{z_j},
  \end{align}
  as measures when $n \to \infty$.  In particular, from convergence
  \eqref{eq:limharmd} it follows that $k \not= 0$ and
  $\sum_{j=1}^k d_j = d$.  Since $|\nabla \hat m_n|^2 (x) = 0$ for all
  $x \not \in \Phi(\overline \Omega)$, we must have
  $z_1,\dots, z_k \in \Phi \left( \overline{\Omega}\right)$.
	
  Pulling back these statements to the plane by precomposing with
  $\Phi$, we obtain the convergence of the exchange energy density to
  the sum of delta measures as in \eqref{eq:8piddxj}. At the same
  time, by estimate \eqref{eq:aniQm12} that applies to $m_n$ we get
  that the rest of the terms in the energy go to zero in the sense of
  measures as $n \to \infty$. This gives the desired result.
\end{proof}

\paragraph{Conflict of interest.} The Authors declare no conflict of
interest.

\paragraph{Data availability.} The manuscript contains no associated
data.





\end{document}